\documentclass{amsart}
\usepackage{geometry, amsthm, amsmath,amsmath,amssymb,amsbsy,graphicx,paralist}
\usepackage{faktor}
\usepackage[mathscr]{euscript}

\numberwithin{equation}{section}

\allowdisplaybreaks[4]

\makeatletter
\renewcommand*\env@matrix[1][*\c@MaxMatrixCols c]{%
 \hskip -\arraycolsep
 \let\@ifnextchar\new@ifnextchar
 \array{#1}}
\makeatother

\theoremstyle{plain}
\newtheorem{theorem}{Theorem}
\newtheorem{corollary}[theorem]{Corollary}
\newtheorem{proposition}[theorem]{Proposition}
\newtheorem{lemma}[theorem]{Lemma}

\newtheorem{question}[theorem]{Question}

\theoremstyle{definition}
\newtheorem{definition}[theorem]{Definition}
\newtheorem{example}[theorem]{Example}

\theoremstyle{remark}
\newtheorem{remark}[theorem]{Remark}

\DeclareMathOperator{\rank}{rank}

\newcommand{\bd}{\partial}

\newcommand{\RR}{\mathbb{R}}

\newcommand{\ZZ}{\mathbb{Z}}

\newcommand{\mO}{\mathcal{O}}

\newcommand{\al}{\alpha}
\newcommand{\be}{\beta}

\newcommand{\la}{\lambda}
\newcommand{\La}{\Lambda}
\newcommand{\Om}{\phi}
\newcommand{\Id}{\textrm{Id}}
\newcommand{\bR}{\overline{R}}
\newcommand{\ds}{\displaystyle}
\DeclareMathOperator{\im}{Im}
\DeclareMathOperator{\coker}{coker}
\DeclareMathOperator{\sign}{sign}

\begin{document}

\title{Directed rooted forests in higher dimension}

\author{Olivier Bernardi}
\address{Department of Mathematics, Brandeis University}
\email{bernardi@brandeis.edu}
\author{Caroline J.\ Klivans}
\address{Departments of Applied Mathematics and Computer Science, Brown University}
\email{klivans@brown.edu}
\date{\today}

\begin{abstract}
For a graph $G$, the generating function of rooted forests, counted by the number of connected components, can be expressed in terms of the eigenvalues of the graph Laplacian. We generalize this result from graphs to cell complexes of arbitrary dimension. This requires generalizing the notion of rooted forest to higher dimension. We also introduce orientations of higher dimensional rooted trees and forests.  These orientations are discrete vector fields which lead to open questions concerning expressing homological quantities combinatorially. 
\end{abstract}
\maketitle

\section{Introduction}
The celebrated matrix-tree theorem gives a determinantal formula for the number of spanning trees of a graph. Specifically, for a graph $G$, the number $\tau(G)$ of spanning trees is expressed in terms of the Laplacian matrix $L_G$ by:
\begin{equation} 
\tau(G)=\det(L_G^v), \label{eq:tree-graph} 
\end{equation}
where $L_G^v$ is the matrix obtained from $L_G$ by deleting the row and column corresponding to a vertex~$v$ (any vertex gives the same determinant). One of the known generalizations of the matrix-tree theorem is a counting formula for \emph{rooted forests} (acyclic subgraph with one marked vertex per component). Namely, the generating function of the rooted forests of a graph $G$, counted by the number of connected components, is 
\begin{equation} 
\sum_{F \subseteq G} x^{\kappa(F)}=\det(L_G+x\,\Id)=\prod_{\lambda} (x + \lambda), \label{eq:forest-graph} 
\end{equation} 
where the sum is over all rooted forests of $G$, the product is over the eigenvalues of $L_G$, and $\kappa(F)$ denotes the number of connected components of the forest $F$. This formula appears, for instance, in \cite{Spectra-of-graphs}, where it is credited to Kelmans.
Note that~\eqref{eq:forest-graph} implies in particular that the number of \emph{rooted} spanning trees of $G$ is 
$$[x^1]\det(L_G+x\,\Id)=\sum_{v\in V}\det(L_G^v)=|V|\det(L_G^v),$$
hence the number of unrooted spanning trees is $\det(L_G^v)$, thereby recovering~\eqref{eq:tree-graph}.

Building on work of Kalai~\cite{Kalai}, the previous works~\cite{Adin, DKM, DKM2, Lyons} gave generalizations of the matrix-tree theorem~\eqref{eq:tree-graph} for higher dimensional analogues of graphs, namely simplicial complexes; see also the survey~\cite{summary}. More specifically, these works introduced the notion of spanning trees (and more generally maximal spanning forests) for simplicial complexes (and more generally cell complexes). These higher dimensional \emph{spanning trees} are the subcomplexes that satisfy certain properties generalizing the graphical properties of being acyclic, connected, and having one less edge than the number of vertices.  It has been shown that the matrix-tree theorem~\eqref{eq:tree-graph} generalizes to this higher dimensional setting.  An important difference in higher dimensions, however, is that the quantity $\tau(G)$ in~\eqref{eq:tree-graph}  has to be understood as a \emph{weighted count} of spanning trees. Each spanning tree $T$ of $G$ is counted with multiplicity equal to its homology cardinality $|H(T)|$ squared. Kalai was the first to understand the significance of this weighted count~\cite{Kalai}, and in particular he derived the following higher dimensional analogue of Cayley's theorem for the complete simplicial complex $K_{n}^d$ of dimension $d$ on $n$ vertices:
$$\tau(K_{n}^d) = n^{ n-2 \choose d}.$$
No closed formula is known for the \emph{unweighted count} of the spanning trees of $K_{n}^d$.


The first goal of this paper is to generalize~\eqref{eq:forest-graph} to simplicial complexes of arbitrary dimension; see Theorem~\ref{thm:main}. 
This generalization requires defining the higher dimensional analogue of a \emph{rooted forest}. Unlike in the graphical case (i.e. dimension 1), the rooted forests are not disjoint unions of rooted trees. Moreover, the \emph{root} of a forest is a non-trivial structure (essentially an entire co-dimension one tree). Lastly, the enumeration is again weighted by homological quantities. For instance, for the complete complex $K_{n}^d$ our result gives
\begin{equation}\label{eq:complete}
\sum_{(F,R)}|H_{d-1}(F,R)|^2 x^{|R|} = x^{n-1\choose d-1}(x+n)^{ n-1 \choose d}, 
\end{equation}
where the sum is over the pairs $(F,R)$ where $F$ is a forest of $K_n^d$, $R$ is a root of $F$, and $|H_{d-1}(F,R)|$ is the cardinality of the relative homology group (see Section~\ref{sec:rooted-forests} for definitions). We also show that other generalizations of the matrix-tree theorem such as those for directed weighted graphs have analogues in higher dimension.

Our second goal is to define the notion of \emph{fitting orientation} for rooted forests of a complex. In dimension~1, the unique fitting orientation of a rooted forest is the orientation of its edges making each tree oriented towards its root. In higher dimension, a fitting orientation of a forest is a bijective pairing between the facets and the non-root ridges, see Figure~\ref{example-orients-intro}. Interestingly, in dimension greater than~1, there can be several fitting orientations of a rooted forest $(F,R)$. In fact, we show that the number of fitting orientations of $(F,R)$ is at least $|H_{d-1}(F,R)|$, and show how to express $|H_{d-1}(F,R)|$ as a signed count of the fitting orientations of $(F,R)$. We consider, but leave open, the possibility of expressing $|H_{d-1}(F,R)|$ as an unsigned count of certain fitting orientations.

\begin{figure}[h]
\includegraphics[width=.5\linewidth]{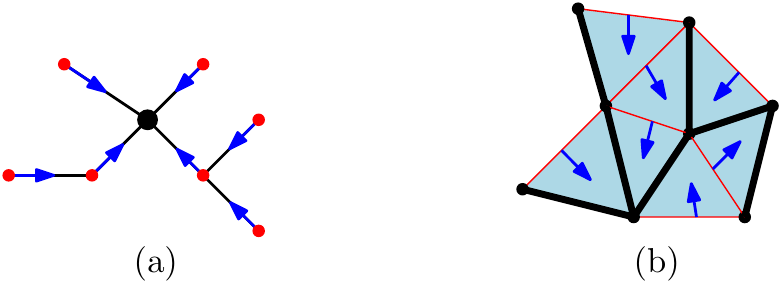}
\caption{The fitting orientations of a graphical rooted tree and a two dimensional rooted tree.}
\label{example-orients-intro}
\end{figure}

Our last goal is of pedagogical nature. We aim to give a self-contained presentation of simplicial trees and forests in a way that makes them more immediately accessible to combinatorialists with little knowledge of homological algebra.

This paper is organized as follows. In Section~\ref{sec:defs} we recall and extend the definitions for forests and roots of simplicial complexes. We also give geometric interpretations of these objects (starting from linear algebra definitions). In Section~\ref{sec:rooted-forests}, we define rooted forests and their the homological weights. In Section~\ref{sec:main-results}, we establish the generalization of~\eqref{eq:forest-graph}. Lastly, in Section~\ref{sec:orientations} we define the fitting orientations of rooted forests, study their properties and conclude with some open questions.

\smallskip


\section{Forests and roots of simplicial complexes}\label{sec:defs}
In this section we review and extend the theory of higher dimensional spanning trees developed in~\cite{Adin, DKM, DKM2, Kalai, Lyons}. 
For simplicity, we will restrict our attention to simplicial complexes throughout, but all results extend more generally to cell complexes.

\subsection{Simplicial complexes}
A \emph{simplicial complex} on $n$ vertices is a collection $G$ of subsets of $[n]$ which is closed under taking subsets, that is, if $g \in G$ and $f \subset g$ then $f \in G$. Elements of $G$ are called \emph{faces}. A \emph{$k$-dimensional face}, or $k$-face for short, is a face $f$ of cardinality $k+1$. The dimension of a complex is the maximal dimension of its faces. 
Observe that \emph{graphs}\footnote{Our \emph{graphs} are finite, undirected, and have neither loops nor multiple edges.} are the same as \emph{1-dimensional simplicial complexes}: the edges are the $1$-dimensional faces, and the vertices are the $0$-dimensional faces.\footnote{Observe that simplicial complexes are special cases of hypergraphs: if the faces of a simplicial complex are seen as \emph{hyperedges} then a simplicial complex is a hypergraph whose collection of hyperedges is closed under taking subsets.}

\begin{figure}[h]
\includegraphics[width=\linewidth]{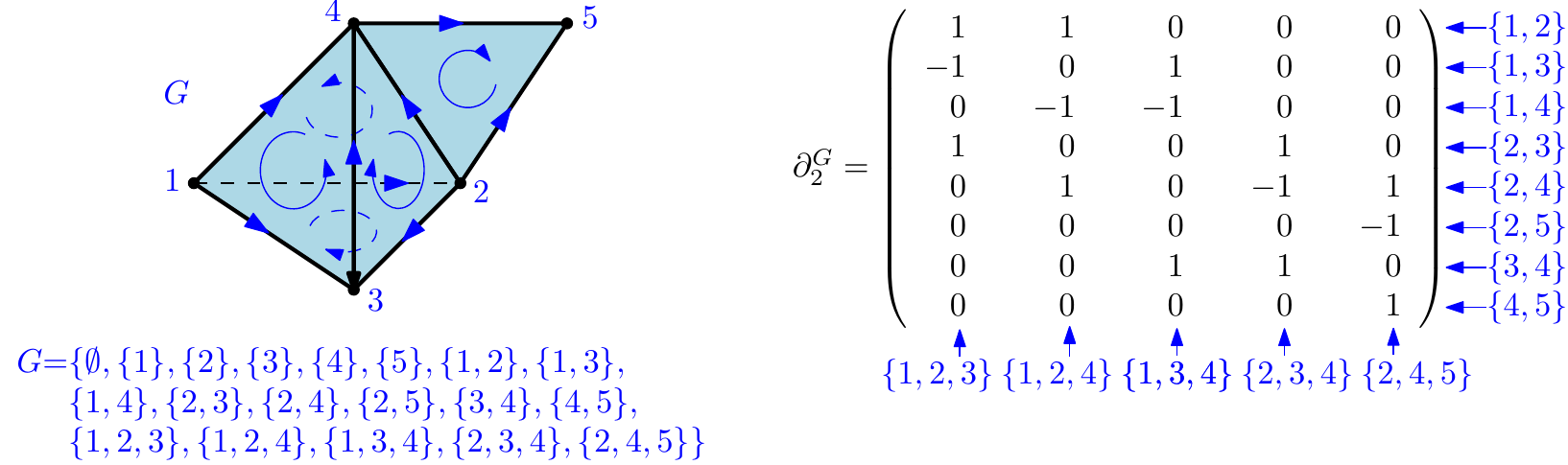}
\caption{A simplicial complex $G$ of dimension 2, and its incidence matrix $\bd_2^{G}$. The standard orientations of the 1-faces and 2-faces are indicated.}
\label{fig:oriented-incidence}
\end{figure}

Let $k\geq 0$. The $k$th \emph{incidence matrix} of a simplicial complex $G$ is the matrix $\bd_k^{G}$ defined as follows: 
\begin{compactitem}
\item the rows of $\bd_k^{G}$ are indexed by the $(k-1)$-faces and the columns are indexed by the $k$-faces, 
\item the entry $\bd_k^{G}[r,f]$ corresponding to a $(k-1)$-face $r$ and a $k$-face $f$ is $(-1)^j$ if $f=\{v_0,v_1,\ldots,v_{k}\}$ with $v_0<v_1\cdots<v_{k}$ and $r=f\setminus \{v_j\}$ for some $j$, and 0 otherwise.
\end{compactitem}
To fix a convention, we will order the rows and columns of $\bd_k^{G}$ according to the lexicographic order of the faces as in Figure~\ref{fig:oriented-incidence} (but all results are independent of this convention).

\begin{remark}
There is a natural matroid structure associated to any simplicial complex $G$. Namely, the \emph{simplicial matroid} $M_G$ of a $d$-dimensional complex $G$ is the matroid associated with the matrix $\bd_d^{G}$: the elements of $M_G$ are the $d$-faces of $G$, and the independent sets of $M_G$ are the sets of $d$-faces such that the corresponding columns of $\bd_d^{G}$ are independent. See~\cite{Cordovil} for more on this class of matroids which generalize graphic matroids.  See also~\cite{Dall} which gives a polyhedral proof of the matrix tree theorem for subclass of regular matroids.
\end{remark}

\begin{remark} Every simplicial complex has a \emph{geometric realization} in which the $k$-faces are convex hulls of $k+1$ affinely independent points. A $0$-face is a point, a $1$-face a segment, a $2$-face a triangle, etc. 
 The complex is a geometric space constructed by gluing together simplices along smaller-dimensional simplices.
 \end{remark}

A complex of dimension 2 is represented in Figure~\ref{fig:oriented-incidence}. On this figure we also show the \emph{standard orientation} of each face, that is, the orientation corresponding to the order of the vertices. The incidence matrix records whether two faces $r,f$ are \emph{incident}: $\bd_k^{G}[r,f]\neq 0$ if $r$ is on the boundary of $f$.  Furthermore, $\bd_k^{G}[r,f]=+1$ if when traversing the boundary of $f$ according to the orientation of $f$, the edge $r$ is traversed in the same direction it is oriented.  Otherwise, $\bd_k^{G}[r,f]=-1$.




\subsection{Forests of a complex}

\begin{definition} \label{def:spanning forest}
Let $G$ be a $d$-dimensional complex, and let $F$ be a subset of the $d$-faces.
We say that $F$ is a \emph{forest} of $G$ if the corresponding columns of $\bd_d^{G}$ are independent.
We say that $F$ is a \emph{spanning subcomplex} of $G$ if the corresponding columns of $\bd_d^{G}$ are of maximal rank.
Accordingly, we say that $F$ is a \emph{spanning forest} of $G$ if the corresponding columns of $\bd_d^{G}$ form a basis of the columns.
\end{definition}

\begin{example} For the complex represented in Figure~\ref{fig:oriented-incidence}, there are 4 spanning forests, which correspond to the 4 ways of deleting one of the triangles on the boundary of the tetrahedron.
\end{example}

\begin{remark}
The forests, spanning subcomplexes, and spanning forests of a complex $G$ correspond respectively to the independent sets, spanning sets, and bases of the simplicial matroid~$M_G$.
\end{remark}

We now discuss how these definitions naturally extend their counterparts in graph theory.
 In the dimension 1 case, 
$F$ is a forest if it is acyclic, and $F$ is spanning if it has the same number of components as $G$. In particular, spanning forests are subgraphs made of one spanning tree per connected component of $G$.

In higher dimensions, let $G_k$ denote the set of $k$-faces of $G$, and $\RR G_k$ the set of formal linear combinations of $k$-faces with coefficient in $\RR$. Note that $\RR G_k$ is a vector space, and that $\bd_k^{G}$ can be interpreted as the matrix of a linear map from $\RR G_k$ to $\RR G_{k-1}$. This map is called the \emph{boundary map} in simplicial topology. It is not hard to show that $\bd_{k-1}^{G}\circ \bd_k^{G}=0$.

\begin{example}
For the complex in Figure~\ref{fig:oriented-incidence}, $\bd_2^{G}(\{1,2,3\})=\{1,2\}-\{1,3\}+\{2,3\}$ and $\bd_2^{G}(\{1,2,3\}-\{1,2,4\}+\{1,3,4\}-\{2,3,4\})=0$. 
\end{example}

A \emph{$k$-cycle} of $G$ is a non-zero element of $\RR G_k$ in the kernel $\ker(\bd_k^{G})$. Informally, a $k$-cycle of $G$ is a combination of $k$-faces such that their boundaries ``cancel out'', as in the preceding example.
 Note that a $1$-cycle is a linear combination of cycles in the usual sense of graph theory, while a 2-cycle is a linear combination of triangulated orientable surfaces without boundary (i.e. spheres, and $g$-tori).
 We will say that a set $S$ of $k$-faces \emph{contains} a $k$-cycle if there exists a $k$-cycle in $\RR S$. 

 By definition, a collection $F$ of $k$-faces of $G$ is a forest if and only if $\ker(\bd_k^{F})=\{0\}$, that is, if $F$ does not contain any $k$-cycle. This is a natural generalization of the condition of being acyclic for graphs. Note here that in taking the boundary $\bd_k^F$, we are implicitly considering the simplicial complex formed from the $k$-faces of $F$ and all possible subsets of the $k$-faces.

Note that for any complex and any $k\geq 1$, the image of an element $x\in \RR G_k$ is a $(k-1)$-cycle because $\bd_{k-1}^{G}\circ \bd_k^{G}=0$. A $(k-1)$-cycle is called a \emph{$G$-boundary} if it is in the image $\im(\bd_k^{G})$, that is, if it is the boundary of a combination of $k$-faces. 


By definition, a collection $F$ of $d$-faces of $G$ is a spanning
subcomplex if and only if $\im(\bd_k^{F})=\im(\bd_k^{G})$, that is, if
any $G$-boundary is an $F$-boundary.

We claim that the above condition is the generalization of being maximally connected for subgraphs. Indeed, observe that in dimension $d=1$, $\bd_0^G(\{v\})=\emptyset$ for every vertex $v$ of $G$. Hence, a 0-cycle is a linear combination of vertices with the sum of coefficients equal to 0. Moreover a 0-cycle is a $G$-boundary if and only if  in each connected component $C$ of $G$, the sum of the coefficients of the vertices in $C$ is equal to 0. Thus, in dimension 1, the $G$-boundaries are all $F$-boundaries if and only if $F$ has the same number of connected components as $G$.

\begin{example} 
If $G$ is a triangulated torus of dimension 2, then the spanning forests are obtained from $G$ by removing any one of the triangles (hence the forests are obtained by removing any non-empty subset of triangles).
If $G$ is a triangulated projective plane (see for instance Figure~\ref{fig:rp2}), then $G$ itself is a spanning forest since it contains no 2-cycle.
\end{example}

\begin{example} Let $G$ be a triangulated sphere of any dimension $d$ (these are the the higher dimensional analogues of cycle graphs).  The spanning forests are obtained by removing exactly one of the $d$-faces, hence the forests of $G$ are obtained from $G$ by removing any non-empty subset of $d$-faces.
\end{example}


The familiar reader will note that the conditions for being a spanning forest can be efficiently expressed in terms of homology.  Spanning forests $F$ of a $d$-dimensional complex $G$ are characterized as the complexes whose maximal faces are the sets of $d$-faces satisfying any two of the three following conditions:
\begin{compactenum}
 \item[(i)] $H_d^{F} = \{0\}$ \hfill ($F$ contains no $d$-cycle),
 \item[(ii)] $\rank H_{d-1}^{F} = \rank H_{d-1}^{G}$ \hfill (any $G$-boundary is an $F$-boundary),
 \item[(iii)] $|F| = |G_d| - \rank(H_d^{G})$ \hfill (the cardinality of $F$ equals the rank of $\bd_d^{G}$).
\end{compactenum}

In dimension 1, a spanning forest is called a \emph{spanning tree} when $G$ is connected. For a $d$-dimensional complex, the condition \emph{$G$ is connected} can be generalized by the condition that any $(d-1)$-cycle is a $G$-boundary (i.e. $\ker(\bd_{d-1}^{G})=\im(\bd_d^{G})$). In this case, following~\cite{DKM}, the spanning forests are also called \emph{spanning trees} of $G$. However, unlike for graphs, spanning forests are not in general disjoint unions of spanning trees.

\subsection{Roots of a complex}
\begin{definition} \label{def:rooting}
Let $G$ be a $d$-dimensional complex. Let $R$ be a subset of the $(d-1)$-faces, and let $\bR=G_{d-1}\setminus R$.
We say that $R$ is \emph{relatively-free} if the rows of $\bd_d^{G}$ corresponding to the faces in $\bR$ are of maximal rank. 
We say that $R$ is \emph{relatively-generating} if the rows of $\bd_d^{G}$ corresponding to the faces in $\bR$ are independent.
We say that $R$ is a \emph{root} of $G$ if it is both relatively-free and relatively-generating, that is, the rows of $\bd_d^{G}$ corresponding to the faces in $\bR$ form a basis of the rows.
\end{definition}

When $G$ is a graph, a subset of vertices $R$ is relatively-free if it contains at most one vertex per connected component, and relatively-generating if it contains at least one vertex per connected component. Accordingly, roots of $G$ are sets consisting of exactly one vertex per connected component.

We now explain how this generalizes to higher dimensions.
First, it is easy to see that $R$ is relatively-free if and only if $\ds \im(\bd_{d}^{G})\cap \RR R = \{0\}$.
Geometrically, this means that $R$ is relatively-free if and only if $R$ contains no $G$-boundary. This generalizes the dimension 1 condition of containing at most one vertex per connected component of $G$.

Second,  it is easy to see that $R$ is relatively-generating if and only if for any $(d-1)$-face $s\in \bR$, there exists $r\in\RR R$ such that $s+r\in \im(\bd_{d}^{G})$. Geometrically, this means that $R$ is relatively-generating if and only if any elements not in $R$ forms a $G$-boundary with $R$. This generalizes the dimension 1 condition of containing at least one vertex per connected component.

 
\begin{example}
Let $G$ be a 2-dimensional complex such that any 1-cycle is a boundary (for instance, a triangulated disc, sphere, or projective plane). In this case, a set $R$ of edges is relatively-free if and only if it does not contain any 1-cycle, that is, if $R$ is a forest of the 1-skeleton of $G$ (its underlying graph). 
Moreover a set $R$ of edges is relatively-generating if any additional edge creates a 1-cycle with $R$, that is, if $R$ connects any pair of vertices. Thus, a set $R$ is a root of $G$ if and only if it is a spanning tree of the 1-skeleton of $G$; see for instance Figure~\ref{fig:rootedforest}~(a) and~(b). 
\end{example}

\begin{example}
Let $G$ be a triangulated 2-dimensional torus. In this case, some 1-cycles are not $G$-boundaries, a set $R$ of edges can be relatively-free even if it contains some cycles. A set $R$ of edges is relatively-free if and only if cutting along the edges of $R$ does not disconnect the torus. Moreover, $R$ is relatively-generating if any additional edge creates a $G$-boundary (i.e. a contractible cycle) with $R$. Hence, $R$ is relatively-generating if and only if cutting along the edges of $R$ gives a disjoint union of triangulated polygons without interior vertices. Thus, $R$ is a root of $G$ if and only if cutting along $R$ gives a triangulated polygon without interior vertices.
\end{example}

\begin{figure}[h]
\includegraphics[width=.9\linewidth]{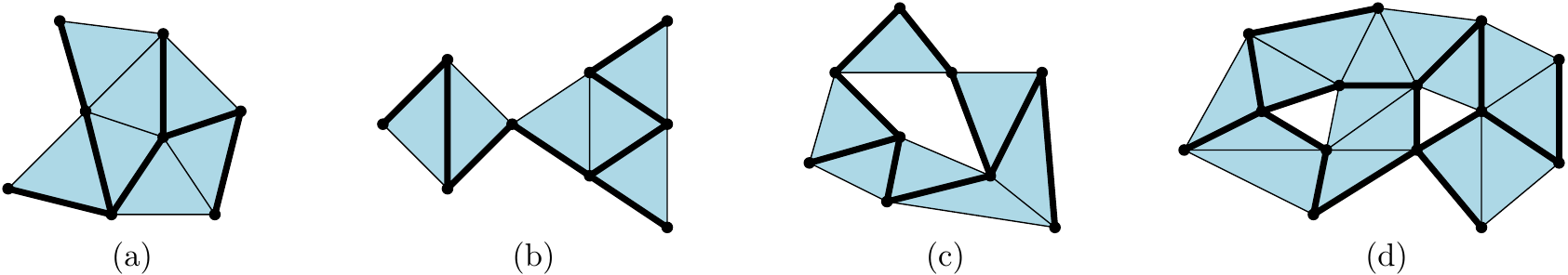}
\caption{Examples of rooted forests in dimension 2 (the root is indicated by thick lines). Observe that in (c) and (d) the root contains some $1$-cycles.}
\label{fig:rootedforest}
\end{figure}

\smallskip




\section{Rooted forests and their homological weights}\label{sec:rooted-forests}
\begin{definition} Let $G$ be a $d$-dimensional simplicial complex. A \emph{rooted forest} of $G$ is a pair $(F,R)$, where $F$ is a forest of $G$, and $R$ is a root of $F$\footnote{Formally, $R$ is a root of the simplicial complex generated by $F$: the subcomplex containing all the faces and subfaces of $F$.}. 
We call $(F,R)$ a \emph{rooted spanning forest} if moreover $F$ is a spanning forest of $G$.
\end{definition}

\begin{example}\label{ex:bipyramid}

  Figure~\ref{fig:rootedforest} shows four examples of rooted forests. In parts (a) and (b), the root is a graphical spanning tree of the one-dimensional faces.  On the other hand, the root in part (c) contains a cycle and the root in part (d) contains two cycles.   Figure~\ref{fig:subforest}~(a) shows the 
  2-dimensional complex $G$ which is the equatorial bipyramid.  Figure~\ref{fig:subforest} also shows a rooted forest of $G$ with three 2-faces. Again, in this case, the root is not an acyclic graph (any root must contain a non-contractible 1-cycle).


  
\end{example}
\begin{figure}[h]
\includegraphics[width=3in]{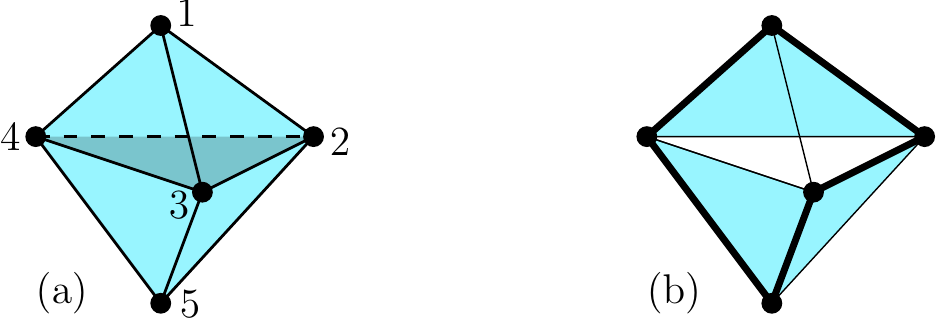}
\caption{(a) The equatorial bipyramid $G$ containing all $7$ triangles supported on the 1-skeleton. (b) A rooted forest $(F,R)$: the forest $F$ is made of the three shaded triangles while the root $R$ is made of the bold edges.}
\label{fig:subforest}
\end{figure}

Next we reproduce two lemmas which already appeared in~\cite{DKM3}.   We include the results here in order to give proofs obtained from our  linear algebra definitions.

\begin{lemma}\cite[Proposition 3.2]{DKM3}\label{lem:non-zero-det}
Let $G$ be a $d$-dimensional simplicial complex, let $F$ be a set of $d$-faces, and let $R$ be a set of $(d-1)$-faces.
Let $\bR=G_{d-1}\setminus R$ and  $\bd_{\bR,F}$ be the submatrix of $\bd_d^{G}$ obtained by keeping the rows corresponding to $\bR$ and the columns corresponding to $F$. Then, $(F,R)$ is a rooted forest if and only if $|F|=|\bR|$ and $\det(\bd_{\bR,F})\neq 0$.
\end{lemma}

\begin{proof}
Suppose first that $(F,R)$ is a rooted forest. Since $F$ is a forest, the columns of $\bd_d^{F}$ are independent. Since $R$ is a root of $F$, the rows of $\bd_d^{F}$ corresponding to $\bR$ form a basis, so that $|F|=|\bR|$ and $\det(\bd_{\bR,F})\neq 0$. 
Suppose now that $|F|=|\bR|$ and $\det(\bd_{\bR,F})\neq 0$. Since $\det(\bd_{\bR,F})\neq 0$, the columns of $\bd_d^{G}$ corresponding to $F$ are independent, hence $F$ is a forest. Since $|F|=|\bR|$ and $\det(\bd_{\bR,F})\neq 0$, the rows of $\bd_d^{F}$ corresponding to $\bR$ form a basis, hence $R$ is a root of $F$. 
\end{proof}

Given a set $R$ of faces of a complex $G$, denote by $G/R$ the cell complex obtained by identifying all the vertices that belong to a face in $R$ (this can be thought of as contracting all the faces of $R$ to a single point). Observe that $\bd_{\bR,G}$ is the incidence matrix $\bd_d^{G/R}$ of the cell complex $G/R$. Moreover, the conditions $|F|=|\bR|$ and $\det(\bd_{\bR,F})\neq 0$ hold if and only if the rows of $\bd_{\bR,G}$ are independent (i.e. $R$ is relatively-generating for $G$) and the columns of $\bd_d^{G/R}$ corresponding to $F$ form a basis (i.e. $F/R$ is a spanning forest of $G/R$). Thus, Lemma~\ref{lem:non-zero-det} gives the following alternative characterization of rooted forests:  $(F,R)$ is a rooted forest of $G$ if and only if $R$ is relatively-generating for $G$ and $F/R$ is a spanning forest of $G/R$.


In order to state the next lemma, we need to introduce the homology groups associated to a rooted forest.
Thus far, we have considered formal linear combinations of faces of $G$ with coefficients in $\RR$. We will now consider the sets $\ZZ G_k$ which are the formal linear combinations of $k$-faces of $G$ with \emph{integer} coefficients. Thus $\ZZ G_k$ is a free abelian group isomorphic to $\ZZ^{|G_k|}$, and the map $\bd_k^G$ is a group homomorphism. We denote $\ds \im_\ZZ(\bd_k^{G})=\bd_k^{G}(\ZZ G_k)$ and $\ds \ker_\ZZ(\bd_k^{G})=\ker(\bd_{k}^{G})\cap \ZZ G_{k}$ the image and kernel of this group homomorphism. 

For all $k$, we have the following inclusion of subgroups:
$$\im_\ZZ(\bd_{k+1}^{G})\subseteq \bd_{k+1}^{G}(\RR G_{k+1})\cap \ZZ G_{k}\subseteq \ker_\ZZ(\bd_{k}^{G}).$$
The \emph{$k$th homology group} is the quotient group 
\begin{equation}\label{eq:def-Hd}
H_k^{G}:=\faktor{\ker_\ZZ(\bd_{k}^{G})}{\im_\ZZ(\bd_{k+1}^{G})}.
\end{equation}
It is easy to see that $H_k^{G}$ is a finite group if and only if $\ker(\bd_{k}^{G})=\im(\bd_{k+1}^{G})$. However it can happen that $H_k^{G}\neq 0$ even if $\ker(\bd_{k}^{G})=\im(\bd_{k+1}^{G})$.

\begin{example} For the triangulated projective plane represented in Figure~\ref{fig:rp2}, we have $\ker(\bd_{1}^{G})=\im(\bd_2^{G})$ (because any 1-cycle is a boundary), but $H_1^{G}\simeq \faktor{\ZZ}{2\ZZ}\neq 0$. This is because, while all 1-cycles of $G$ can be obtained as boundaries of combinations of 2-faces with real coefficients, some 1-cycles (such as $C=\{1,2\}+\{2,3\}-\{1,3\}$) can only be obtained with even multiplicity using integer coefficients. For example, if $S$ denotes the sum of all $2$-faces oriented clockwise in Figure~\ref{fig:rp2}, then $\bd_2(S)=2\cdot C$.
\end{example}
\begin{definition}
Let $G$ be a $d$-dimensional simplicial complex. Given a subset $F$ of $d$-faces and a subset $R$ of $(d-1)$-faces, consider the following group homomorphism
$$\Psi_{F,R}:
\begin{array}[t]{ccl} 
\ZZ F&\to& \faktor{\ZZ G_{d-1}}{\ZZ R},\\
x &\mapsto & \bd_d^{G}(x)+\ZZ R.
\end{array}
$$
The {\emph{relative homology groups}} for the pair $(F,R)$ are given by:
$$
\begin{array}[t]{ccl}
 H_d(F,R) & = & \ker \Psi_{F,R},\\
H_{d-1}(F,R) & = & \coker \Psi_{F,R}:=
\faktor{\left(\faktor{\ZZ G_{d-1}}{\ZZ R}\right)}{\im(\Psi_{F,R})}.
\end{array}
$$
\end{definition}

\begin{lemma}\cite[Proposition 3.5]{DKM3} \label{lem:value-det}
If $(F,R)$ is a rooted forest of $G$, then $\ds |\det(\bd_{\bR,F})|= |H_{d-1}(F,R)|$.
\end{lemma}

\begin{proof} Recall that for any non-singular integer matrix $M$ of dimension $n\times n$ one has
$\ds \left|\det(M)\right|=\left|\faktor{\ZZ^{n}}{L}\right|$ where $L$ is the subgroup of $\ZZ^n$ generated by the columns of $M$. 
Thus $\left|\det(\bd_{\bR,F})\right|=\left|\faktor{\ZZ \bR}{L_{F,R}}\right|$, where $\bd_{\bR,F}$ is the subgroup of $\ZZ \bR$ generated by the columns of $L_{F,R}$. Moreover, the isomorphism
$\ZZ \bR\simeq \faktor{\ZZ G_{d-1}}{\ZZ R}$ gives $L_{F,R}\simeq \im(\Psi_{F,R})$. Thus, 
$$\ds \left|\det(\bd_{\bR,F})\right|=\bigg|\faktor{\ZZ \bR}{L_{F,R}}\bigg|=\bigg|\faktor{\left(\faktor{\ZZ G_{d-1}}{\ZZ R}\right)}{\im(\Psi_{F,R})}\bigg|= \left|H_{d-1}(F,R)\right|.$$\\[-.5cm]
\end{proof}


\begin{remark} 
Readers familiar with algebraic topology will recognize the groups $H_d(F,R)$ and $H_{d-1}(F,R)$ as the relative homology groups of the complexes $F$ and $R$. In the case $|\bR|=|F|$ one has $\det(\bd_{\bR,F})\neq 0$ if and only if $H_d(F,R)=0$. Thus Lemma~\ref{lem:non-zero-det} can be restated as: $(F,R)$ is a rooted forest if and only if $|\bR|=|F|$ and $H_d(F,R)=0$.
\end{remark}

\smallskip

\section{Counting rooted forests}\label{sec:main-results}
In this section we establish a determinantal formula, generalizing~\eqref{eq:forest-graph}, for the generating function of rooted forests enumerated by size.

The \emph{Laplacian matrix} of a simplicial complex $G$ of dimension $d$ is the matrix 
$$L_G= \bd_d^{G}\cdot (\bd_d^{G})^T.$$ 
The rows and columns of $L_G$ are indexed by $(d-1)$-faces. 
For any $(d-1)$-face $r$, the entry $L_{G}[r,r]$ is equal to the number of $d$-faces containing $r$. 
Moreover, for any $(d-1)$-face $s\neq r$, the entry $L_{G}[r,s]$ is 0 if $r$ and $s$ are not incident to a common $d$-face, and is $(-1)^{i+j}$ if there exists a $d$-face $f=\{v_0,\ldots,v_d\}$ with $v_0<\cdots <v_d$ such that $r=f\setminus\{v_i\}$ and $s=f\setminus\{v_j\}$.

\begin{theorem}\label{thm:main} 
For a $d$-dimensional simplicial complex $G$, the characteristic polynomial of the Laplacian matrix $L_G$ gives a generating function for the rooted forest of $G$. More precisely,
\begin{equation}\label{eq:forest-complex}
\sum_{(F,R) \textrm{ rooted forest of }G} |H_{d-1}(F,R)|^2 x^{|R|}= \det(L_G+x\cdot \Id),
\end{equation}
where $\Id$ is the identity matrix of dimension $|G_{d-1}|$.

Formula~\eqref{eq:forest-complex} can also be enriched by attributing weights to the faces in the forest and its root:
if $\{x_r\}_{r\in G_{d-1}}$ and $\{y_f\}_{f\in G_{d}}$ are indeterminates indexed by the $(d-1)$-faces and $d$-faces of $G$ respectively, then
\begin{equation}\label{eq:forest-complex2}
\sum_{(F,R) \textrm{ rooted forest of }G} |H_{d-1}(F,R)|^2 \prod_{r\in R}x_r\prod_{f\in F}y_f= \det(L_G^y+X),
\end{equation}
where $L_G^y=\bd_d^{G}\cdot Y\cdot (\bd_d^{G})^T$, and where $X$ (resp. $Y$) is the diagonal matrix whose rows and columns are indexed by $G_{d-1}$ (resp. $G_{d}$), and whose diagonal entry in the row indexed by $r$ is $x_r$ (resp. indexed by $f$ is $y_f$).
\end{theorem}

\begin{remark} Equations \eqref{eq:forest-complex} and \eqref{eq:forest-complex2} are closely related to Equations (7) and (8) of \cite{Martin:Pseudodeterminants-tree-counts} giving a general formula for the determinant of the form $\det(M Y M^T+x\cdot \Id)$ were $M$ is an arbitrary rectangular matrix and $Y$ is a diagonal matrix. Indeed, \eqref{eq:forest-complex} and \eqref{eq:forest-complex2} can be obtained simply by combining Equations (7) and (8) of \cite{Martin:Pseudodeterminants-tree-counts} with  Lemmas~\ref{lem:non-zero-det} and~\ref{lem:value-det}. 
\end{remark}

\begin{proof} 
It suffices to prove~\eqref{eq:forest-complex2} since it implies~\eqref{eq:forest-complex} by setting $x_r=x$ for all $r$ and $y_f=1$ for all $f$.
We first define two matrices $M,N$ whose rows are indexed by $G_{d-1}$ and whose columns are indexed by $G_{d}\cup G_{d-1}$.
The matrix $M$ (resp. $N$) is the matrix obtained by concatenating the matrix $\bd_d^{G}\cdot Y$ (resp. $\bd_d^{G}$) with the matrix $X$ 
(resp. $\Id$). Observe that $L_G^y+X=M\cdot N^T$. Now, applying the Binet-Cauchy formula gives
 $$\det(L_G^y+X)=\det(M\cdot N^T) = \sum_{S\subseteq G_{d} \cup G_{d-1},~ |S|=|G_{d-1}|} \det(M_{S})\det(N_{S}),$$
where $M_S$ and $N_S$ are the submatrices of $M$ and $N$ obtained by selecting only the columns indexed by the set $S$.
Let us fix $S=F\cup R$ with $F\subseteq G_d$, $R\subseteq G_{d-1}$, and $|S|=|G_{d-1}|$. Let $\bd_{\bR,F}$ (resp. $\bd_{\bR,F}^y$) be the submatrix of $\bd_d^{G}$ (resp. $\bd_d^{G}\cdot Y$) obtained by keeping the rows corresponding to $\bR$ and the columns corresponding to $F$. By rearranging the rows of $M_S$ so that the bottom rows correspond to the elements in $R$ we obtain a matrix with upper-left block $\bd_{\bR,F}^y$, upper-right block 0, and lower-left block equal to the diagonal matrix $X_R$ (with rows and column indexed by $R$) with diagonal entry $x_r$ in the row indexed by $r\in R$. Thus, 
$$\det(M_S)=\pm \det(\bd_{\bR,F}^y)\prod_{r\in R}x_r=\pm \det(\bd_{\bR,F})\prod_{r\in R}x_r\prod_{f\in F}y_f,$$
where the second equality comes from the fact that $\bd_{\bR,F}^y=\bd_{\bR,F}\cdot Y_F$, where $Y_F$ is the diagonal matrix (with rows and columns indexed by $F$) with diagonal entry $y_f$ in the row indexed by $f\in F$. 
Since $N$ is obtained from $M$ by setting $x_r=1$ for all $r$ and $y_f=1$ for all $f$, we get $\det(N_S)=\pm \det(\bd_{\bR,F})$ (with the same sign as in $\det(M_S)$). Hence 
$$\det(M_S)\det(N_S)= {\det(\bd_{\bR,F})}^2\prod_{r\in R}x_r\prod_{f\in F}y_f,$$
and
$$\det(L_G^y+X)=\sum_{ F\subseteq G_d, ~R\subseteq G_{d-1},~|F|+|R|=|G_{d-1}|} {\det(\bd_{\bR,F})}^2\prod_{r\in R}x_r\prod_{f\in F}y_f.$$
By Lemma~\ref{lem:non-zero-det} only the pairs $(F,R)$ corresponding to rooted forests of $G$ contribute to the above sum, and for these pairs Lemma~\ref{lem:value-det} gives $\det(\bd_{\bR,F})^2=|H_{d-1}(F,R)|^2$. This gives~\eqref{eq:forest-complex2}.
\end{proof}

 Theorem~\ref{thm:main} implies the Simplicial Matrix Tree Theorem derived in~\cite{Adin, DKM, DKM2, DKM3, Kalai, Lyons} which generalizes the usual matrix-theorem~\eqref{eq:tree-graph} to higher dimension.
\begin{corollary}[Simplicial matrix-tree theorem]
Let $G$ be a simplicial complex and let $R$ be a root of $G$. Then 
$$\sum_{F \textrm{ spanning forest of } G}|H_{d-1}(F,R)|^2=\det(L_G^R),$$
where $L_G^R$ is the submatrix of $L_G$ obtained by deleting the rows and columns of $L_G$ corresponding to faces in $R$.
\end{corollary}

\begin{proof} Note that $R$ is a root of $G$ if and only if it is the root of
some, equivalently all, spanning forests of $G$. Thus, extracting the coefficient of $\ds \prod_{r\in R}x_r$ in~\eqref{eq:forest-complex2} and setting $y_f=1$ for all $f$ gives the corollary.
\end{proof}

\begin{example}

  For the 2-dimensional complex $G$ represented in Figure~\ref{fig:subforest}, we have
$$\det(L_G+x\cdot \Id) = x^4(x+5)^3(x+3)^2 = x^9 + 21x^8 + 174x^7 + 710x^6 + 1425x^5 + 1125x^4.$$

  Moreover, for any rooted forest $(F,R)$ of $G$ we have $H_{1}(F,R)=\{0\}$. Thus Theorem~\ref{thm:main} gives
  $$\sum_{(F,R) \textrm{ rooted forest of }G} x^{|R|}=x^9 + 21x^8 + 174x^7 + 710x^6 + 1425x^5 + 1125x^4.$$

The complex has 1125 rooted spanning forests (which are in fact rooted spanning trees). There are 15 spanning forests of $G$ because one obtains a spanning forest by either removing the equator face $\{2,3,4\}$ and any other 2-face (6 possibilities), or keeping the equator face and removing a 2-face from the top and a 2-face from the bottom (9 possibilities). For each of these spanning trees, there are $75$ possible roots (which are the spanning trees of the 1-skeleton of $G$), hence a total of $1125= 15 \times 75$ rooted spanning forests.  This gives the coefficient of $x^4$.  

In general, the coefficient of $x^k$ gives the number of rooted forests with $9-k$ triangles (since for any rooted forest $|F|+|R|=|G_{d-1}|=9$).
\end{example}

\begin{example}
For the triangulation of the real projective plane represented in Figure~\ref{fig:rp2}, we have
$$\begin{array}{lll}\det(L_G+x\cdot \Id)&=&x^5(x+3)^4(x+3+ \sqrt{5})^3(x+3- \sqrt{5})^3\\
&=&x^{15} + 30x^{14} + 390x^{13} + 2880x^{12} + 13305x^{11} + 39906x^{10}\\&&+ 78040x^9 + 97320x^8 + 73440x^7 + 30240x^6 + 5184x^5.
\end{array}
$$
In this case, $G$ is itself a spanning forest (actually a spanning tree). The possible roots of $G$ correspond to the spanning trees of the 1-skeleton of $G$. The 1-skeleton of $G$ is the complete graph $K_6$, which has $6^4$ spanning trees, hence there are $6^4 = 1296$ possible roots for $G$, and they have cardinality 5. Finally, for each root $R$ of $G$, the homology group $H_{1}(G,R)$ has order 2. Accordingly the coefficient of $x^5$ is $2^2 \cdot 1296 = 5184$.
\end{example}

\begin{example}
Let $K_{n}^d$ be the complete complex of dimension $d$ on $n$ vertices. 
The eigenvalues of $L_{K_n^d}$ are 0 with multiplicity ${n-1\choose d-1}$ and $n$ with multiplicity ${ n-1 \choose d}$.
Thus, Theorem~\ref{thm:main} gives 
\begin{equation}\label{eq:complete-forests} 
\sum_{(F,R) \textrm{ rooted forest of }K_n^d}|H_{d-1}(F,R)|^2\, x^{|R|} = x^{n-1\choose d-1}(x+n)^{n-1 \choose d}. 
\end{equation}
Observe that some of the weights $|H_{d-1}(F,R)|$ are greater than 1. For instance, the triangulated projective plane $F$ represented in Figure~\ref{fig:rp2} is one of the spanning forests of $K_6^2$, and for any rooting $R$ of $F$ we have $|H_{1}(F,R)|=2$.
\end{example}

\begin{example}
Let $C_{n}^d$ be the hypercube on $2^n$ vertices in dimension $d$. It was shown in~\cite{DKM2} that the non-zero eigenvalues of $L_{C_n^d}$ are $2j$ for $j\in \{d,d+1,\ldots,n\}$ with multiplicity ${j-1 \choose d-1}{n \choose j}$. Thus, Theorem~\ref{thm:main} gives 
\begin{equation}\label{eq:hypercube-forests} 
\sum_{(F,R) \textrm{ rooted forest of }C_n^d}|H_{d-1}(F,R)|^2\, x^{|R|} = x^{{n\choose d-1}2^{n+1-d}}\prod_{j=d}^n \left(1+\frac{2j}{x}\right)^{{j-1 \choose d-1}{n \choose j}}.
\end{equation}
\end{example}

At present, both~\eqref{eq:complete-forests} and~\eqref{eq:hypercube-forests} have  combinatorial proofs only for $d=1$ (see~\cite{Prufer:Cayley,Bernardi}). Finding a combinatorial proof for $d\geq 2$ represents a major challenge in bijective combinatorics. The first difficulty is to interpret the term $|H_{d-1}(F,R)|$ on the left-hand side combinatorially. This is one of the motivations for the notion of fitting orientation presented in the next section.

\smallskip

\section{Orientations of rooted forests}\label{sec:orientations}
In this section we define \emph{fitting orientations} for rooted forests.
 Fitting orientations are a generalization of the \emph{standard orientation} of a rooted graphical forest in which each tree is oriented towards its root-vertex. Namely, this is the unique orientation of the edges of the forest such that every non-root vertex has outdegree 1.


\begin{definition}
Let $G$ be a simplicial complex of dimension $d$. A \emph{fitting orientation} of a rooted forest $(F,R)$ is a bijection $\Om$ between $\bR:=G_{d-1}\setminus R$ and $F$, such that each face $r\in \bR$ is mapped to a face $f\in F$ containing $r$. A \emph{bi-directed rooted forest} of $G$ is a triple $(F, R, \Om,\Om')$ such that $(F,R)$ is a rooted forest and $\Om,\Om'$ are fitting orientations.
\end{definition}

Observe that for a graph $G$, the unique fitting orientation of a rooted forest $(F,R)$ is the bijection $\Om$ which associates an incident vertex $v$ to each edge $e=\{u,v\}\in F$, where $v$ is the child of $u$ in the rooted forest. 
This unique fitting orientation naturally identifies with the standard orientation of $(F,R)$ defined above.

An important difference in dimension $d>1$, is that there can be several fitting orientations for a rooted forest. For instance, Figure~\ref{fig:rp2} shows the two fitting orientations for a rooted forest of the projective plane.
\begin{figure}[h]
\includegraphics[width=3.5in]{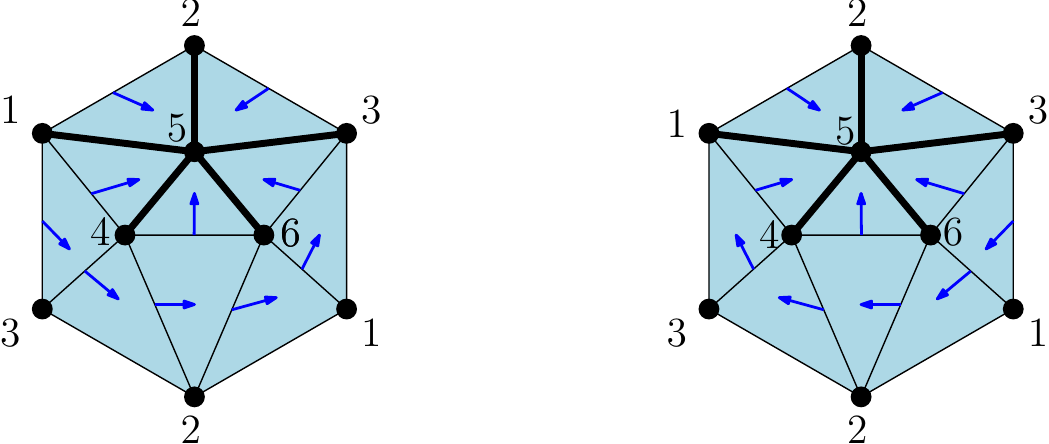}
\caption{The two fitting orientations of a rooted forest $(F,R)$ of a triangulated projective plane. Here $F$ consists of every triangle and $R$ is indicated in bold lines.}
\label{fig:rp2}
\end{figure}

\begin{remark}
Fitting orientations are a special case of  discrete vector fields from discrete Morse theory~\cite{Forman}. A
\emph{discrete vector field} is a collection of pairs of faces where each pair $(r,f)$ is such that:
(1) $\dim(r) = \dim(f)-1$, (2) $r \subset f$, (3) every face of the complex is in at most one pair.
Fitting orientations are not in general discrete \emph{gradient} vector fields, as gradient vector fields are forbidden to contain \emph{closed paths of faces} (e.g. the bold pairings for the 5 triangles at the bottom of the projective plane example). In \cite{Forman:det-laplacian} Forman actually considers vector fields on graphs in relation to the graph Laplacian  with respect to a representation of the fundamental group of $G$.    

Discrete vector fields do not require that the face pairings are only between faces in the top two dimensions.
One natural extension of a fitting orientation is to recursively root and orient trees
down in dimension. For example, for the projective plane, the root is a graphical spanning tree, and this tree can be given a fitting orientation. Such recursive orientations are then more similar to the bijections
used in simplicial decomposition theorems for $f$-vector characterizations, see for example~\cite{Duval, Rstan2}.
\end{remark}

We now prove a lower bound for the number of fitting orientations of rooted forests.
\begin{proposition}\label{prop:number-orientations}
Let $(F,R)$ be a rooted forest of a simplicial complex $G$. The number of fitting orientations of $(F,R)$ is greater than or equal to the homological weight $|H_{d-1}(F,R)|$. In particular, any rooted forest admits a fitting orientation.
\end{proposition} 


\begin{proof}
By Lemma~\ref{lem:value-det}, $|H_{d-1}(F,R)|=|\det(\bd_{\bR,F})|$, where $\bd_{\bR,F}$ is the submatrix of $\bd_d^{G}$ with rows corresponding to $\bR=G_{d-1}\setminus R$ and columns corresponding to $F$. Let us denote by $n=|F|=|\bR|$ the size of the matrix $\bd_{\bR,F}$, and by $\bd_{\bR,F}[i,j]$ its coefficient in row $i$ and column $j$. We have 
$$\det(\bd_{\bR,F})=\sum_{\pi\in S_n} \sign(\pi)\prod_{i=1}^n\bd_{\bR,F}[i,\pi(i)],$$
where the sum is over all the permutation of $[n]$ and $\sign(\pi)$ is the sign of the permutation $\pi$. Now we identify the permutations $\pi\in S_n$ with the bijections from $\bR$ to $F$ (via the ordering of faces in $\bR$ and $F$ by lexicographic order), and we get 
\begin{equation}\label{eq:sum-orient}
\det(\bd_{\bR,F})=\sum_{\Om:\bR\to F~bijection} \sign(\Om)\prod_{r\in \bR}\bd_{d}^{G}[r,\Om(r)],
\end{equation}
where $\sign(\Om)$ is the sign of the permutation $\pi$ identified to $\Om$, and $\bd_{r,f}$ is the coefficient of the boundary matrix $\bd_d^{G}$ corresponding to the $(d-1)$-face $r$ and $d$-face $f$. Moreover, since $\bd_{r,f}$ is 0 unless $r$ is contained in $f$, the only bijections $\Om$ contributing to the sum in~\eqref{eq:sum-orient} are the fitting orientations of $(F,R)$. This gives 
$$|H_{d-1}(F,R)|=\bigg|\sum_{\Om\in\mO} \sign(\Om)\prod_{r\in \bR}\bd_{d}^{G}[r,\Om(r)] \bigg|\leq \sum_{\Om\in\mO} \bigg|\sign(\Om)\prod_{r\in \bR}\bd_{d}^{G}[r,\Om(r)] \bigg|=|\mO|,$$
where $\mO$ is the set of fitting orientations of $(F,R)$.
\end{proof}

\begin{remark}
It is tempting to conjecture that the number of fitting orientations is equal to $|H_{d-1}(F,R)|$. This is however not the case as Figure~\ref{dunce} illustrates: the rooted forest is topologically a $3$-fold dunce cap and has 3 fitting orientations, whereas $|H_{1}(F,R)|=1$.
\end{remark}

\begin{figure}[h]
\includegraphics[width=\linewidth]{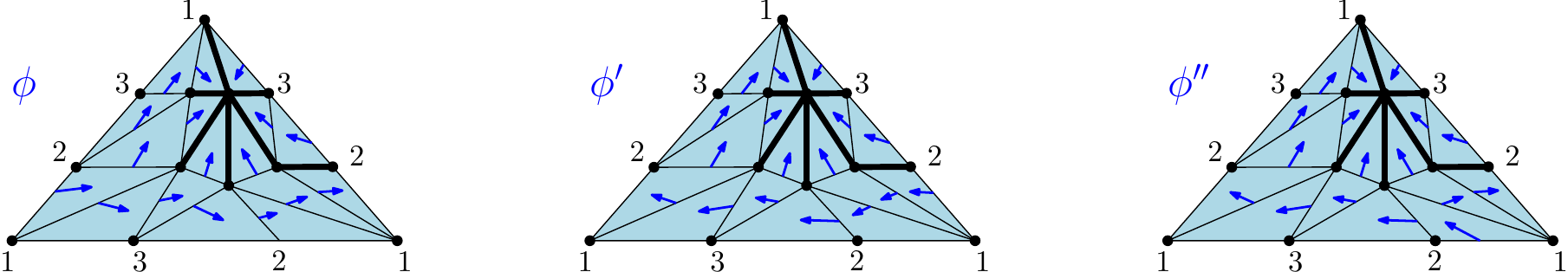} \label{dunce}
\caption{The three fitting orientations for the rooted forest $(F,R)$, where $F$ contains all the triangles (hence is a triangulation of the $3$-fold
 dunce cap) and $R$ is made of the bold edges.}
\end{figure}



We now define the sign of a fitting orientation.
Let $(F,R)$ be a rooted forest of a $d$-dimensional complex $G$, and let $n=|F|=|\bR|$. Let $\al$ be a linear ordering of $\bR$, and $\be$ a linear ordering of $F$. For any bijection $\phi:\bR\to F$, we associate the unique permutation $\pi_{\al,\be,\phi}\in S_n$ such that $\phi$ maps the $i$th face in $\bR$ (in the $\al$ order) to the $\pi_{\al,\be,\phi}(i)$th face in $F$ (in the $\be$ order).

For any fitting orientation $\Om$ of $(F,R)$, define 
$$\la_{\al,\be}(\Om)=\sign(\pi_{\al,\be,\Om})\prod_{r\in \bR} \bd_{d}^{G}[r,\Om(r)].$$
Define $\sign(\al,\be)\in\{-1,1\}$ as the sign of the determinant of the matrix $\bd_{\bR,F}^{\al,\be}$ obtained by reordering $\bd_{\bR,F}$ so that the rows are ordered according to $\al$ and the columns are ordered according to $\be$. 
Finally, define the \emph{sign} of the fitting orientation $\Om$ as 
\begin{equation}\label{eq:sign-orient}
\La_{\al,\be}(\Om)=\sign(\al,\be)\cdot\la_{\al,\be}(\Om).
\end{equation}

\begin{proposition}\label{prop:homology-as-orient}
For any rooted forest $(F,R)$ and any linear ordering $\al$ of $F$ and $\be$ of $\bR$,
\begin{equation}\label{eq:homology-as-orient}
\sum_{\Om\textrm{ fitting orientation of }(F,R)}\La_{\al,\be}(\Om)=|H_{d-1}(F,R)|.
\end{equation}
Moreover, $\La_{\al,\be}(\Om)\in\{-1,1\}$ does not depend on $(\al,\be)$, and we abbreviate it to $\La(\Om)$.
\end{proposition}

\begin{remark}
Previously we showed that the number of fitting orientations of a rooted forest $(F,R)$ is at least the size of the homology group $H_{d-1}(F,R)$. Proposition~\ref{prop:homology-as-orient} shows how fitting orientations ``cancel down in pairs'' to precisely $|H_{d-1}(F,R)|$; see the related Question~\ref{question:1} below.
\end{remark}


\begin{proof}
Using Lemma~\ref{lem:value-det}, and expanding the determinant gives
$$\begin{array}{lll}
|H_{d-1}(F,R)|&=&\sign(\al,\be)\det(\bd_{\bR,F}^{\al,\be})\\
&=&\sign(\al,\be)\sum_{\pi\in S_n}\sign(\pi)\prod_{i=1}^n\bd_{\bR,F}^{\al,\be}[i,\pi(i)]\\
&=&\sign(\al,\be)\sum_{\phi:\bR \to F\textrm{ bijection}}\sign(\pi_{\al,\be,\phi})\prod_{r\in \bR}\bd_{r,\phi(r)}\\
&=&\sum_{\Om \textrm{ fitting orientation of }(F,R)}\La_{\al,\be}(\Om).
\end{array}
$$
We now show that $\La_{\al,\be}(\Om)$ does not depend on $\al$, $\be$. Let $\al'$ be a linear ordering of $\bR$, and let $\be'$ be a linear ordering of $F$. Let $\sigma$ be the permutation mapping $\al$ to $\al'$ and $\theta$ be the permutation mapping $\be$ to $\be'$. Then, for any fitting orientation $\phi:F\to\bR$, one has $\pi_{\al',\be',\Om}=\theta\circ \pi_{\al,\be,\Om}\circ \sigma^{-1}$, so that $\la_{\al',\be'}(\Om)=\sign(\sigma)\sign(\theta)\la_{\al,\be}(\Om)$. Moreover, $\sign(\al',\be')=\sign(\sigma)\sign(\theta)\sign(\al,\be)$, hence $\La_{\al',\be'}(\Om)=\La_{\al,\be}(\Om)$.
\end{proof}


Using Proposition~\ref{prop:homology-as-orient} one can reinterpret Theorem~\ref{thm:main} in terms of \emph{bi-directed rooted forests}: in place of a homologically weighted enumeration of rooted forest, one gets a signed enumeration of bi-directed rooted forests. 

\begin{theorem}\label{thm:counting-oriented-forests}
Let $G$ be a $d$-dimensional simplicial complex, and let $L_G$ be its Laplacian matrix. Then, 
\begin{equation}\label{eq:directed-forest}
\sum_{(F,R,\Om,\Om')\textrm{ bi-directed rooted forest}}\La(\Om)\La(\Om')\,x^{|R|}=\det(L_G+x\cdot \Id).
\end{equation}
More generally, given indeterminates $\{w_{r,f}\}_{r\in G_{d-1},f\in G_d}$, define the \emph{weighted Laplacian matrix} as $L_{G,w}=\bd_d^{G,w}\cdot (\bd_d^{G})^T$, where $\bd^{G,w}_d$ is the matrix obtained from $\bd^{G}_d$ by multiplying by $w_{r,f}$ the entry corresponding to the faces $r$ and $f$. Then, for any indeterminates $\{x_{r}\}_{r\in G_{d-1}}$,
\begin{equation}\label{eq:directed-forest2}
\sum_{(F,R,\Om,\Om')\textrm{ bi-directed rooted forest}}\La(\Om)\La(\Om')\left(\prod_{r\in \bR}w_{r,\Om(r)}\right) \left(\prod_{r\in R}x_r\right)=\det(L_{G,w}+X),
\end{equation}
where $X$ is the diagonal matrix whose rows and columns are indexed by $G_{d-1}$, and whose diagonal entry in the row indexed by $r$ is $x_r$.
\end{theorem}

\begin{remark} 
Note that in the special case where $w_{r,f}=y_f$ for all $r\subset f$,~\eqref{eq:directed-forest2} gives 
$$\sum_{(F,R,\Om,\Om')\textrm{ bi-directed rooted forest}}\La(\Om)\La(\Om')\prod_{r\in R} x_{r}\prod_{r\in \bR} y_{f}=\det(L_{G,w}+X)=\det(L_{G}^y+X),$$
which implies~\eqref{eq:forest-complex2} via Proposition~\ref{prop:homology-as-orient}.
Another special case of \eqref{eq:directed-forest2} corresponds to setting $w_{r,f}\in \{0,1\}$ for all $r\subseteq f$. This has the effect of restricting which pairs of incident faces are allowed in the fitting orientations. In this case, the matrix $L_{G,w}$ can be thought as the Laplacian for a \emph{directed} simplicial complex (the case $d=1$ gives the Laplacian of a directed graph). 
For general indeterminates $\{w_{r,f}\}$, \eqref{eq:directed-forest2} generalizes to higher dimensions the \emph{forest extension} of the matrix-tree theorem for \emph{weighted directed graphs}.
\end{remark}

\begin{proof}[Proof of Theorem~\ref{thm:counting-oriented-forests}]
This is similar to the proof of Theorem~\ref{thm:main}.
Define two matrices $M,N$ whose rows are indexed by $G_{d-1}$ and whose columns are indexed by $G_{d}\cup G_{d-1}$: the matrix $M$ (resp. $N$) is obtained by concatenating the matrix $\bd^{G,w}_d$ (resp. $\bd_d^{G}$) with the matrix $X$ (resp. $\Id$). Observe that $L_{G,w}+X=M\cdot N^T$.
The Binet-Cauchy formula gives
$$\det(L_{G,w}+X)=\det(M\cdot N^T) = \sum_{S\subseteq G_{d} \cup G_{d-1},\atop |S|=|G_{d-1}|} \det(M_{S})\det(N_{S}),$$
where $M_S$ and $N_S$ are the submatrices of $M$ and $N$ obtained by selecting only the columns indexed by the set $S$.
Let us fix $S=F\cup R$ with $F\subseteq G_d$, $R\subseteq G_{d-1}$, and $|S|=|G_{d-1}|$. 
It is easy to see that $\det(M_S)=\pm \det(\bd_{\bR,F}^w)\prod_{r\in R}x_r$, where $\bd_{\bR,F}^w$ is the submatrix of $\bd_d^{G,w}$ with rows corresponding to $\bR$ and columns corresponding to $F$. By expanding the determinant we get
$$\begin{array}{lll}\ds\det(M_S)&=&\ds \pm\sum_{\Om\textrm{ fitting orientation of }(F,R)}\la_{\al,\be}(\Om)\prod_{r\in \bR} w_{r,\Om(r)}\prod_{r\in R}x_r\\
&=&\ds \pm\sum_{\Om\textrm{ fitting orientation of }(F,R)}\La(\Om)\left(\prod_{r\in \bR}w_{r,\Om(r)}\right) \left(\prod_{r\in R}x_r\right),
\end{array}
$$
where $\al$ and $\be$ are the lexicographic order for $\bR$ and $F$.
Now since $N_S$ is obtained from $M_S$ by setting $x_r=1$ and $w_{r,f}=1$ for all $f,r$, we have
$$\det(M_{S})\det(N_{S})=\sum_{\Om,\Om'\textrm{ fitting orientations of }(F,R)}\La(\Om)\La(\Om')\left(\prod_{r \in \bR}w_{r,\Om(r)}\right) \left(\prod_{r\in R}x_r\right)$$
if $(F,R)$ is a rooted forest, and 0 otherwise (because in this case $\det(N_S)=\det(\bd_{\bR,F})=0$ by Lemma~\ref{lem:non-zero-det}).
This completes the proof of~\eqref{eq:directed-forest2}, hence also of~\eqref{eq:directed-forest}.
\end{proof}

We now give another, more natural, combinatorial expression for the sign of a bi-directed forest. Given a bi-directed forest $(F,R,\Om,\Om')$ we consider the permutation $\theta:=\Om^{-1}\circ\Om'$ of $\bR$. The fixed points of $\theta$ are the $(d-1)$-faces $r\in \bR$ for which the two orientations coincide: $\Om(r)=\Om'(r)$. The other cycles $(r_1,r_2,\ldots,r_k)$ of $\theta$ can be interpreted geometrically as encoding a closed path on the faces in $\bR$ and $F$ of the form $\ds r_1,f_1,r_2,f_2,\ldots,r_k,f_k,r_{k+1}=r_1,$ 
where $\Om'(r_i)=f_i=\Om(r_{i+1})$ for all $i\in\{1\ldots k\}$.
Such a cycle of $\theta$ is called \emph{oriented strip of $(\Om,\Om')$} if 
$$\ds 
\prod_{i=1}^k\bd_{r_i,f_i}\bd_{r_{i+1},f_i}=(-1)^{k}.
$$ 
See Figure~\ref{fig:dunce-strips} for an example. The name ``oriented strip'' reflects the fact that in dimension $d=2$, an oriented strip of $(\Om,\Om')$ encodes a closed path on the faces forming a cylinder, whereas the other (non-fixed point) cycles of $\theta$ encode closed paths forming a M\"obius strip.

\begin{figure}[h]
\includegraphics[width=4in]{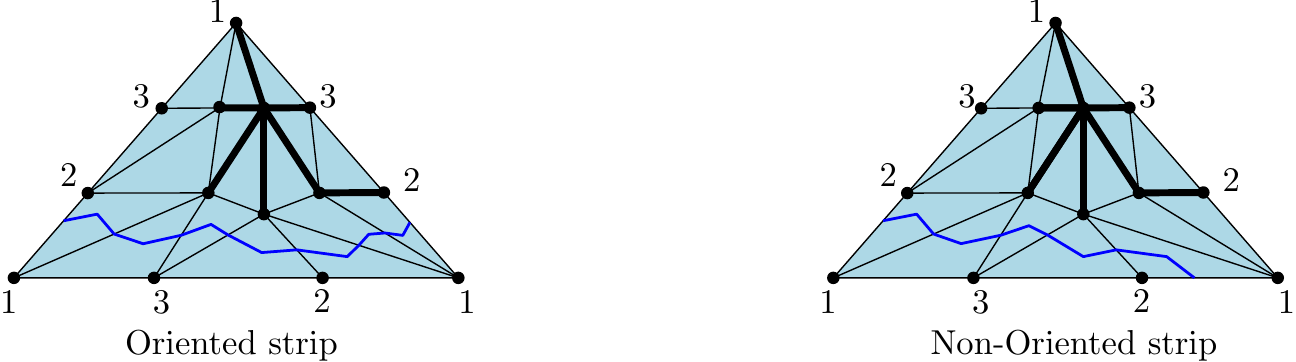}\label{fig:dunce-strips}
\caption{An oriented strip of $(\Om,\Om')$ and a non-oriented strip of $(\Om,\Om'')$, where $\Om$, $\Om'$, and $\Om''$ are the three fitting orientations of the rooted forest represented from left to right in Figure~\ref{dunce}.}
\end{figure}

\begin{proposition}\label{prop:sign-biorientation}
For any bi-directed forest $(F,R,\Om,\Om')$, 
$$\ds \La(\Om)\La(\Om')=(-1)^{\#\textrm{oriented strips of }(\Om,\Om')}.$$
This gives new expressions for \eqref{eq:directed-forest} and \eqref{eq:directed-forest2}. For instance,
$$\sum_{(F,R,\Om,\Om')\textrm{ bi-directed rooted forest}}(-1)^{\#\textrm{oriented strips of }(\Om,\Om')}x^{|R|}=\det(L_G+x\cdot \Id).$$
\end{proposition}

\begin{proof} Let $\al,\be$ be linear orderings of $\bR$ and $F$ respectively.
$$\begin{array}{lll}
\ds\La(\Om)\La(\Om')
=\la_{\al,\be}(\Om)\la_{\al,\be}(\Om')
&=&\ds\sign(\pi_{\al,\be,\Om})\sign(\pi_{\al,\be,\Om'})\prod_{r\in \bR}\bd_{r,\Om(r)}\bd_{r,\Om'(r)}\\
&=&\ds\sign(\Om^{-1}\circ\Om')\prod_{r\in \bR}\bd_{r,\Om(r)}\bd_{r,\Om'(r)}\\
&=&\ds \prod_{C \textrm{ cycle of }\Om^{-1}\circ\Om'}\left((-1)^{|C|-1}\prod_{r\in C}\bd_{r,\Om(r)}\bd_{r,\Om'(r)}\right).
\end{array}
$$
where the second identity uses $\Om^{-1}\circ\Om'=\pi_{\al,\be,\Om}\circ\pi_{\al,\be,\Om'}$. Moreover, the contribution of oriented strips of $\Om^{-1}\circ\Om'$ to the above product is~$-1$, and the contribution of the other cycles of $\Om^{-1}\circ\Om'$ is~$1$.
\end{proof}

We conclude this paper with some open questions.

\begin{question}\label{question:1} Is there a combinatorial way of defining a subset of the fitting orientations of a rooted forest $(F,R)$ which is equinumerous to $|H_{d-1}(F,R)|$?
One way to achieve this goal would be to define a (partial) matching of the fitting orientations of $(F,R)$ such that any matched pair of orientations $(\Om,\Om')$ satisfies $\La(\Om)\La(\Om')=-1$, and any unmatched orientation $\Om$ satisfies $\La(\Om)=1$. 
\end{question}

\begin{question}\label{question:2} Is there a combinatorial way of defining a subset of the pairs of fitting orientations $(\Om,\Om')$ of a rooted forest $(F,R)$ which is equinumerous to $|H_{d-1}(F,R)|^2$? Again, this could be defined in terms of a partial matching on the set of pairs $(\Om,\Om')$. 
\end{question}
Note that an answer to Question~\ref{question:1} would immediately give as answer to  Question~\ref{question:2}. In view of Proposition~\ref{prop:sign-biorientation}, Question~\ref{question:2} may appear more natural.

\begin{question} The notion of fitting orientations can be extended to pairs $(F,R)$ with $|F|=|\bR|$ which are not rooted forests. Is there a direct combinatorial proof of Lemma~\ref{lem:non-zero-det} in terms of fitting orientation? More precisely, when $(F,R)$ is not a rooted forest, we would like to find a \emph{sign reversing involution} implying
$$\sum_{\Om\textrm{ fitting orientation of }(F,R)}\la_{\al,\be}(\Om)=0.$$
The dimension 1 case is easy, but the higher dimensional cases seem more challenging.
\end{question}

\noindent {\bf Acknowledgments.} We thank Vic Reiner for pointing out several useful references. OB acknowledges the partial support provided by the NSF grant DMS-1400859.

\bibliographystyle{plain} 
\bibliography{biblio-simplicial-forests}

\end{document}